\documentclass{article}
\usepackage{amsmath}
\usepackage{amsfonts}
\usepackage{amsthm}
\usepackage{color}
\usepackage{enumerate}
\usepackage{amssymb}

\newfont{\rus}{wncyr10 scaled 1200}
\newfont{\rusb}{wncyb10 scaled 1200}

\newtheorem{theorem}{Theorem}[section]

\newtheorem{question}[theorem]{Question}
\newtheorem{lemma}[theorem]{Lemma}

\newtheorem{corollary}[theorem]{Corollary}

\theoremstyle{definition}
\newtheorem{definition}[theorem]{Definition}
\theoremstyle{remark}
\newtheorem{remark}[theorem]{Remark}

\newcommand{\conv}{{\rm conv}\hskip0.02cm}
\newcommand{\supp}{{\rm supp}\hskip0.02cm}

\newcommand{\ep}{\varepsilon}

\newcommand\remove[1]{}

\begin{document}
\title{\LARGE{\bf  Weak$^*$ closures and derived sets for convex sets in dual Banach spaces}}

\author{Mikhail I.~Ostrovskii}

\date{\today}
\maketitle

\begin{large}

\noindent{\bf Abstract:} The paper is devoted to the convex-set
counterpart of the theory of weak$^*$ derived sets initiated by
Banach and Mazurkiewicz for subspaces. The main result is the
following: For every nonreflexive Banach space $X$ and every
countable successor ordinal $\alpha$, there exists a convex subset
$A$ in $X^*$ such that $\alpha$ is the least ordinal for which the
weak$^*$ derived set of order $\alpha$ coincides with the weak$^*$
closure of $A$. This result extends the previously known results
on weak$^*$ derived sets by Ostrovskii (2011) and Silber (2021).
\medskip

\medskip

\medskip

\section{Introduction}

Let $X$ be a Banach space. For a subset $A$ of the dual Banach
space $X^*$, we denote the weak$^*$ closure of $A$ by
${\overline{A}\,}^*$. The {\it weak$^*$ derived set} of $A$ is
defined as
\[A^{(1)} = \bigcup_{n=1}^\infty {\overline{A\cap
nB_{X^*}}\,}^{*},\] where $B_{X^*}$ is the unit ball of $X^*$.
That is, $A^{(1)}$ is the set of all limits of weak$^*$ convergent
bounded nets in $A$. If $X$ is separable, $A^{(1)}$ coincides with
the set of all limits of weak$^*$ convergent sequences from $A$,
called the {\it weak$^*$ sequential closure}. The strong closure
of a set $A$ in a Banach space is denoted by $\overline{A}$. We
set $A^{(0)}:= A$.

It was noticed in the early days of Banach space theory by
Mazurkiewicz \cite{Maz30} that $A^{(1)}$ does not have to coincide
with ${\overline{A}\,}^*$ even for a subspace $A$, and
$(A^{(1)})^{(1)}$ can be different from $A^{(1)}$. In this
connection, it is natural to introduce derived sets for all ordinal
numbers as: (1) If $A^{(\alpha)}$ has been already defined, then
$A^{(\alpha+1)}:=(A^{(\alpha)})^{(1)}$; (2) If $\alpha$ is a limit
ordinal and $A^{(\beta)}$ has been already defined for all
$\beta<\alpha$, then
\begin{equation}\label{E:DefLimDer}
A^{(\alpha)}:=\bigcup_{\beta<\alpha}A^{(\beta)}.\end{equation}

The study of weak$^*$ derived sets was initiated by Banach and
Mazurkiewicz (see \cite{Maz30,Ban32}). Its early developments are
discussed at length in the Appendix to the classical monograph by
Banach \cite{Ban32}. Later, this study was continued by many
authors and found significant applications. Since the well-known
survey \cite{PB79} in the fields of Banach space theory initiated
by \cite{Ban32} does not mention developments stemmed from
Banach's ``Appendix'' \cite[Annexe]{Ban32}, it looks beneficial to
present here a short historical account.

Banach and Mazurkiewicz were primarily interested in the case of a
separable Banach space $X$. Banach asked whether the weak$^*$
sequential closure of a subspace does not have to be weak$^*$
sequentially closed, and Mazurkiewicz \cite{Maz30} gave an
affirmative answer to this question. This result was the reason
for Banach to introduce weak$^*$ sequential closures of all
transfinite orders.

In \cite[Annexe]{Ban32} Banach proved that weak$^*$ sequential
closures of finite orders do not have to be weak$^*$ sequentially
closed. Furthermore, Banach stated that in his paper, which was
going to appear in {\it Studia Math.}, volume {\bf 4}, he proved a
similar result for $X=c_0$ and an arbitrary countable ordinal.
Nevertheless, the cited paper has never been published. A possible
explanation of this situation can be that Banach found a mistake
in his proof when it was late to delete the statement and the
reference in \cite{Ban32}. It is regrettable that the story was
left uncommented in the reprint of \cite{Ban32} in \cite{Ban79}
and the survey \cite{PB79} because the editors of \cite{Ban79}
might have known the actual story.

In the late 20s and early 30s, Banach and his school focused on
the sequential approach to weak$^*$ topology and did not use the
notion of weak$^*$ topology. The subject of what is now called
General Topology already existed, see \cite{AU24}, but was not yet
well known. Using General Topology significant part of the theory
was made more elegant (see an account in \cite{DS58}). However,
the sequential approach developed in \cite[Annexe]{Ban32} has its
advantages and has led to significant applications.  An early
application of weak$^*$ sequential closures to the study of sets
of uniqueness for Fourier series was discovered by
Piatetski-Shapiro \cite{Pia54}, and further developed in
\cite{KL87} and \cite{Lyo88}.

As for further development of theoretical aspects of weak$^*$
sequential closures, it is worth pointing out that Banach's claim
mentioned above (see \cite[Annexe, \S1]{Ban32}) was proved in 1968
by McGehee \cite{McG68}, using results by Piatetski-Shapiro
\cite{Pia54}. At the same time, Sarason \cite{Sar68, Sar68b}
proved similar results for some other spaces.

Davis and Johnson \cite{DJ73} developed an essential tool for
investigation of non-quasi-reflexive Banach spaces (that is,
spaces for which the canonical image of $X$ in $X^{**}$ has
infinite codimension). This tool was used by Godun \cite{God78} to
prove that for any finite ordinal the dual of any
non-quasi-reflexive Banach space contains a subspace whose
weak$^*$ sequential closures of finite orders form a strictly
increasing sequence (this result was rediscovered later by
Moscatelli \cite{Mos87}). Godun \cite{God77} also made an attempt
to prove similar results for infinite countable ordinals, but his
argument in \cite{God77} contains gaps.

The result which completes the investigation for separable Banach
spaces was proved  in \cite{Ost87} for general non-quasi-reflexive
separable Banach spaces. Namely, it was proved that, for every
separable non-quasi-reflexive Banach space $X$ and every countable
ordinal $\alpha$, the space $X^*$ contains a linear subspace
$\Gamma$ for which $\Gamma^{(\alpha)}\ne\Gamma^{(\alpha+1)}=X^*$.
This result completes the investigation for separable Banach
spaces for the following reasons: (1) It is easy to see that if
$X$ a separable quasi-reflexive Banach space and $\Gamma$ is a
subspace of $X^*$, then $\Gamma^{(1)}={\overline{\Gamma}\,}^*$.
(2) It is known (a proof is sketched in \cite[page~213]{Ban32})
that for a separable Banach space $X$ and any convex subset
$A\subset X^*$, there is a countable ordinal $\alpha$ such that
$A^{(\alpha)}={\overline{A}\,}^*$.

This result of \cite{Ost87} is presented in \cite{HMVZ07}.
Regrettably, the historical information on weak$^*$ sequential
closures in \cite{HMVZ07} is inaccurate.

Meanwhile, the theory of weak$^*$ sequential closures found
applications in many different fields:

\begin{itemize}

\item The structure theory of Fr\'echet spaces
(\cite{BDH86,DM87,MM89,Mos90,Ost98});

\item The Borel and Baire classification of linear operators,
including applications to the theory of ill-posed problems
(\cite{Sai76,PP80,Pli97,Raj04});

\item  The operator theory (\cite{Sar72});

\item The theory of universal Markushevich bases
(\cite{Pli86,HMVZ07}).

\end{itemize}

The survey \cite{Ost01} contains a more detailed historical
account on weak$^*$ sequential closures, which was up-to-date in
2000.

It is worth mentioning that the proof of nonexistence of universal
Markushevich bases in \cite{Pli86} uses the existence of subspaces
satisfying $\Gamma^{(\alpha)}\ne\Gamma^{(\alpha+1)}=X^*$ in the
same way as Szlenk \cite{Szl68} uses the existence of reflexive
spaces with an arbitrary large Szlenk index in his proof of
nonexistence of universal reflexive Banach spaces.
\medskip

Recently, sequential closures and derived sets became objects of
interest in some other areas,  such as:

\begin{itemize}

\item Extension problems for holomorphic functions on dual Banach
spaces (\cite{GKM10});

\item Valuations (\cite{Ale17});

\item The mathematical economics (\cite{DO19});

\item The duality operators/spaces (\cite{Sal21+}).
\end{itemize}

Garcia-Kalenda-Maestre \cite{GKM10} initiated the theory of
weak$^*$ derived sets for convex sets. This theory was developed
by Ostrovskii \cite{Ost11}, who proved that, for an arbitrary
non-reflexive Banach space $X$ (not necessarily
non-quasi-reflexive), there a convex set $A\subset X$ for which
$A^{(1)}\ne A^{(2)}$. This result showed that the theory for
convex sets is different from the theory for subspaces - for
convex sets quasi-reflexivity does not imply that
$A^{(1)}={\overline{A}\,}^*$. Silber \cite{Sil21} developed this
result further by proving that, for an arbitrary non-reflexive
space $X$ and an arbitrary $n\in\mathbb{N}$, the space $X^*$
contains a convex subset $A$ for which $n$ is the least ordinal
satisfying $A^{(n)}={\overline{A}\,}^*$ and a convex subset $D$
for which $\omega+1$  is the least ordinal satisfying
$D^{(\alpha)}={\overline{D}\,}^*$ . Our goal is to develop this
theory further by proving the following result.

\begin{theorem}\label{T:Successor} Let $X$ be a non-reflexive Banach space and $\alpha$ be a countable ordinal. Then there
exists a convex subset $A\subset X^*$ such that $\alpha$ is the
least ordinal for which $A^{(\alpha)}\ne
A^{(\alpha+1)}={\overline{A}}^*$.
\end{theorem}

This theorem is proved in Section \ref{S:PrfThrm}.\medskip

\section{Trees and Forests}\label{S:Trees&Forests}

Our construction of sets is quite different from the one in
\cite{Ost11, Sil21}. For our construction of sets whose existence
is stated in Theorem \ref{T:Successor}, we need families of both
finite and infinite forests - that is, graphs with no cycles, and
trees - that is, connected forests. It has to be pointed out
 that although our terminology and constructions are
related to the ones in \cite[p.~161]{Die90}, \cite[Section
2]{JO98}, \cite{Kec95}, \cite{KL87}, \cite[Section 2]{Ode04}, and
\cite{Tro92}, they are different.

To achieve the goal of this work, we need a family of graphs
labelled by countable ordinals $\alpha$, including all finite
ordinals. The graph structure of these graphs is induced by a
partial order on their vertex sets. To be specific, they are
constructed using the definition below.

\begin{definition}[Graph structure on partially ordered sets]\label{D:GrParOrd} Given  be a nonempty partially ordered set $V$ such that the set $\{y:y\le x\}$ is finite and linearly ordered
for each $x\in V$. The
 graph $F_V$ is defined in the following way. The vertex set of $F_V$ is
taken to be $V$. Two vertices $x,y \in F_V$ are {\it adjacent} if
and only if  they are  comparable and  the set of vertices between
them is empty. For a vertex $x\in F_V,$  its \emph{up-degree} is
defined as the cardinality of the set of vertices $y$ adjacent to
$x$ and satisfying $x<y$. Each such $y$ is called an
\emph{up-neighbor} of $x$, which is written as $y\succ x$.
Likewise, the \emph{down-degree} of a vertex $x\in F_V$ is defined
as the cardinality of the set of vertices $y$ adjacent to $x$ and
satisfying $y<x$. Each such $y$ is called a \emph{down-neighbor}
of $x$, and we write $y\prec x$.
\end{definition}

If the context specifies a set $V$, we omit the lower index of
$F_V$ and denote it by $F$ whenever it does not lead to
misunderstanding.

It is clear that the obtained graph $F$ is a forest and that the
down-degree of a vertex in $F$ can be either $0$ or $1$, while the
up-degree can be any cardinal. Vertices with down-degree $0$ are
called {\it initial}. Vertices with up-degree $0$ are called {\it
terminal}.

For each ordinal $\alpha<\omega_1$ ($\omega_1$ is our notation for
the least uncountable ordinal), we construct a forest $F_\alpha$ of
the type described above. The forest $F_\alpha$ will be called a
\emph{forest of order $\alpha$}. Our construction of $F_\alpha$ is
inductive.

Throughout the paper, all ordinals (except $\omega_1$) are assumed
to be countable, that is, satisfying $<\omega_1$.

The forest $F_0$ of order $0$ contains one vertex, no edges.
Suppose that we have already defined the forest $F_\alpha$ of
order $\alpha$. To get the forest $F_{\alpha+1}$ of order
$\alpha+1$ in the case where $\alpha$ is a successor ordinal, we
introduce $\aleph_0$ (notation for the cardinality of
$\mathbb{N}$) disjoint copies of $F_\alpha$, by this we mean that
any two elements of different copies are incomparable in the
partial order, and each copy inherits the partial order of
$F_\alpha$. After that, we add to the vertex set one more vertex
as a vertex which is $<$ each of the vertices of each copy of
$F_\alpha$ and define $F_{\alpha+1}$ as the forest obtained using
Definition \ref{D:GrParOrd} for the resulting partially ordered
set. Observe that the added vertex is adjacent to the initial
vertex of each copy of $F_\alpha$, and has no other adjacent
vertices.

If $\beta<\omega_1$ is a limit ordinal, we consider a strictly
increasing sequence $\{\beta_n\}_{n=1}^\infty$ of successor
ordinals converging to $\beta$ and define $F_\beta$ as a disjoint
union of $F_{\beta_n}$ over all $n\in \mathbb{N}.$ This means that
we use Definition \ref{D:GrParOrd} for the partial order on the
union of $F_{\beta_n}$'s in which each $F_{\beta_n}$ keeps its
partial order, and vertices of different $F_{\beta_n}$'s  are not
comparable. Therefore, in this case $F_\beta$ is disconnected with
$\aleph_0$ connected components.
\medskip

Finally, to define $F_{\alpha+1}$ in the case where $\alpha$ is a
limit ordinal, we add to the graph $F_\alpha$ one more vertex
which is $<$ each of the vertices of $F_\alpha$ and use Definition
\ref{D:GrParOrd}. The definition of $F_\alpha$ implies that this
new vertex has up-degree $\aleph_0$. \medskip

Observe that up-degrees of vertices in the forests $\{F_\alpha\}$
can be only $0$ or $\aleph_0$. This statement about up-degrees
follows by observing that each new - that is, not a copy of an
introduced before - vertex which is used in the construction of
$F_\alpha$ for $\alpha\ge 1$ has up-degree $\aleph_0$. Only
vertices obtained from the vertex forming $F_0$ by repeated
copying have up-degree $0$.\medskip

Note that $F_n$ for a finite ordinal $n$ is what is called the
(infinitely) \emph{countably branching tree of depth} $n$.
\medskip

For a forest $F$ of the described above type, we define the
\emph{derived forest} $F^{1}$ as a subgraph of $F$ from which we
delete all infinite sets of terminal vertices having the same
down-neighbor. If we have already defined the derived forest
$F^{\beta}$ of order $\beta$, we define the derived forest of
order $\beta+1$ as $(F^{\beta})^{1}$. If $\beta$ is a limit
ordinal and we have already defined derived forests for smaller
ordinals, we set
$F^{\beta}=\cap_{\gamma<\beta}F^{\gamma}$.\medskip

Note that the derived forest $(F_\alpha)^{\gamma}$ (we write
$F_\alpha^{\gamma}$ since it does not create any confusion) of
order $\gamma$ is not necessarily of the form $F_\beta$ for some
ordinal $\beta$. For example, $F_{\omega_0}^{\omega_0}$
($\omega_0$ is the least infinite ordinal) is a collection of
isolated vertices of cardinality $\aleph_0$.
\medskip

To obtain a  description of derived forests, we start with the
following construction. If $\{K_i\}_{i=1}^k$ are forests
constructed from partially ordered sets as in Definition
\ref{D:GrParOrd}, by $\biguplus_{i=1}^k K_i$ we denote the
disjoint union of $\{K_i\}_{i=1}^k$, which can also be described
as a graph constructed using Definition \ref{D:GrParOrd} for the
partial order on the union of $\{K_i\}_{i=1}^k$ which on sets
$K_i$ coincides with their original partial order, and any two
elements of different $K_i$ are incomparable. This notation is
also used in the case where $k=\infty$. We also shall use the
notation

\[\left(\biguplus_{i=1}^k K_i\right)\biguplus\left(\biguplus_{i=k+1}^\infty
K_i\right),\] which is self-explanatory at this point.

The following two lemmas contain results on derived forests needed
for our construction.

\begin{lemma}\label{L:Der} {\bf (1)} Let $n$ and $k$, $1\le k\le n$, be finite
ordinals, then $(F_n)^k=F_{n-k}$.
\medskip

{\bf (2)} Let $\alpha$ be any ordinal. Then
$(F_{\alpha+1})^\alpha=F_1$ and $(F_{\alpha+1})^{\alpha+1}=F_0$.
\medskip

{\bf (3)} Let $\alpha$ be a limit ordinal and
$\{\beta_n\}_{n=1}^\infty$ be an increasing sequence of successor
ordinals converging to $\alpha$, used to define $F_\alpha$. Then,
for $\beta<\alpha$,
\begin{equation}\label{E:AlphaBeta}(F_\alpha)^\beta=\left(\biguplus_{n,~\beta\ge\beta_n}
F_0\right)\biguplus\left(\biguplus_{n,~\beta<\beta_n}
(F_{\beta_n})^\beta\right),
\end{equation}
and
\[(F_\alpha)^\alpha=\biguplus_{i=1}^\infty F_0.\]

\end{lemma}

\begin{proof} {\bf (1)} The case of finite ordinals follows easily
from the definition of $F_n$.

It is easy also to see from the definition of $F_\alpha$ for a
successor ordinal that the conditions $(F_{\alpha+1})^\alpha=F_1$
and $(F_{\alpha+1})^{\alpha+1}=F_0$ imply
$(F_{\alpha+2})^{\alpha+1}=F_1$ and
$(F_{\alpha+2})^{\alpha+2}=F_0$.

\medskip

Before proving {\bf (2)} for limit ordinals (provided we proved it
for all smaller ordinals), we need to prove the equality
\eqref{E:AlphaBeta}.

To get equality \eqref{E:AlphaBeta} we use the obvious
$F_0^\gamma=F_0$ for every $\gamma$, its generalization
\[\left(\biguplus_{n,~\beta\ge\beta_n} F_0\right)^\gamma
=\left(\biguplus_{n,~\beta\ge\beta_n} F_0\right),\] and the
definition of $F_\alpha$. The equality \eqref{E:AlphaBeta} implies
that $F_{\alpha+1}^\alpha=F_1$ and $F_{\alpha+1}^{\alpha+1}=F_0$
for a limit ordinal $\alpha$, provided the statement {\bf (2)} is
known for all ordinals $\beta<\alpha$.

This implies all formulas stated in Lemma \ref{L:Der}.
\end{proof}

\begin{lemma}\label{L:UpTerm} If a vertex in $(F_\alpha)^\beta$ has
infinitely many terminal up-neighbors, then all of its
up-neighbors are terminal vertices.
\end{lemma}

\begin{proof} We prove this using induction on $\alpha$. For
finite $\alpha$, the statement immediately follows from the
description of $F_\alpha^\beta$ in Lemma \ref{L:Der} {\bf(1)}.

If the statement is true for all ordinals which are strictly less
that a limit ordinal $\alpha$, then it is true for $F_\alpha$. In
fact, since $F_\alpha$ consists of components which are $F_\beta$
for $\beta<\alpha$, and derived forests of disconnected forests
are taken componentwise, the conclusion follows.

Now we assume that the statement is true for $F_\alpha$, and
derive it for $F_{\alpha+1}$. There are two cases:

\begin{enumerate}[{\bf (a)}]

\item $\alpha$ is a limit ordinal,

\item $\alpha$ is a successor ordinal,
\end{enumerate}

In both cases the induction hypothesis implies that the condition
holds for all vertices of all derived forest except, possibly the
initial vertex.\medskip

\noindent{\bf Case (a):} For the initial vertex it is also
satisfied because Lemma \ref{L:Der} {\bf (3)} implies that for
$\beta<\alpha$, the initial vertex has finitely many terminal
up-neighbors in the derived forest $F_{\alpha+1}^\beta$. On the
other hand, all of the up-neighbors of the initial vertex are
terminal in the derived forest $(F_{\alpha+1})^\alpha$.
\medskip

\noindent{\bf Case (b):} In this case, each of the derived forests
consists of $\aleph_0$ copies of the same tree attached to the
initial vertex. The conclusion follows.
\end{proof}

\section{Reduction to separable case}\label{S:RedSeparable}

Our main goal is to prove Theorem \ref{T:Successor}. The main part
of our proof of Theorem \ref{T:Successor} is its proof for a
separable Banach space with a basis satisfying special condition.
In this section we show that this special case implies the general
case of Theorem \ref{T:Successor}.\medskip

We need the following notation. Let $Z$ be a closed subspace in a
Banach space $X$ and $E:Z\to X$ be the natural isometric
embedding. Then $E^*:X^*\to Z^*$ is a quotient mapping which maps
each functional in $X^*$ onto its restriction to $Z$. Let $A$ be a
subset of $Z^*$. It is clear that $D=(E^*)^{-1}(A)$ is the set of
all extensions of all functionals in $A$ to the space $X$.

\begin{lemma}[{cf. \cite{Ost87,Ost11}}]\label{L:DerExt}
For any ordinal $\alpha$ the equality
\begin{equation}\label{E:DerExt}
D^{(\alpha)}=(E^*)^{-1}(A^{(\alpha)})
\end{equation}
holds, where the derived set $D^{(\alpha)}$ is taken in $X^*$ and
the derived set $A^{(\alpha)}$ -- in $Z^*$.
\end{lemma}

\begin{proof} The inclusion $D^{(\alpha)}\subset(E^*)^{-1}(A^{(\alpha)})$ follows from
the weak$^*$ continuity of the operator $E^*$ using transfinite
induction.
\medskip

To  prove  the  inverse inclusion by transfinite induction, it
suffices to show that for every bounded net $\{f_\nu\}\subset
Z^{*}$ with $w^*-\lim_\nu f_\nu=f$ and every $g\in
(E^{*})^{-1}(\{f\})$ there exist $g_{\nu}\in
(E^{*})^{-1}(\{f_\nu\})$ such that some subnet of $\{g_\nu\}$ is
bounded and weak$^*$ convergent to $g$. Let $h_{\nu}$ be such that
$h_{\nu}\in (E^{*})^{-1}(\{f_{\nu}\})$ and $||h_{\nu}||=
||f_{\nu}||$ (Hahn-Banach extensions). Then $\{h_{\nu}\}_\nu$ is a
bounded net in $X^{*}$. Hence it has a weak$^*$ convergent subnet,
let $h$ be its limit. Then $g-h\in(E^{*})^{-1}(\{0\})$, therefore
$g_\nu=h_\nu+g-h$ is a desired net.
\end{proof}

\begin{proof}[Reduction of Theorem \ref{T:Successor} to its special case] We are going to
use the following result proved in \cite{Pel62,Sin62}: If a Banach
space $X$ is non-reflexive, then it contains a bounded basic
sequence $\{z_i\}_{i=1}^\infty$ such that $\|z_i\|\ge 1$ for every
$i\in \mathbb{N}$, but
\begin{equation}\label{E:BddPartSum} \sup_{1\le k<\infty}\left\|\sum_{i=1}^kz_i\right\|=C<\infty.\end{equation}
Let $Z$ be the closed linear span of the sequence
$\{z_i\}_{i=1}^\infty$ and $\{z_i^*\}_{i=1}^\infty\subset Z^*$ be
the biorthogonal functionals of $\{z_i\}_{i=1}^\infty$. Let
$z^{**}$ be a weak$^*$-cluster point of the sequence
$\left\{\sum_{i=1}^kz_i\right\}_{k=1}^\infty$ in $Z^{**}$.

We will need the following observations about these vectors:

\begin{enumerate}[{\bf (a)}]

\item\label{I:Le} $|z^{**}(x)|\le C\|x\|$ for every $x\in Z^*$.
This is an immediate consequence of $\|z^{**}\|\le C$ which
follows from \eqref{E:BddPartSum}.

\item\label{I:Ge} If $x$ is a linear combination of
$\{z_i^*\}_{i=1}^\infty$ with nonnegative coefficients, then
$z^{**}(x)\ge c\|x\|$ for some $c>0$.

In fact, let $x=\sum_{i=1}^ka_iz_i^*$. Then
$z^{**}(x)=\sum_{i=1}^k a_i$. On the other hand,
\[\|x\|=\left\|\sum_{i=1}^ka_iz_i^*\right\|\le\sup_i\|z_i^*\|\,\sum_{i=1}^k
a_i.\]

Since $\{z_i\}$ is a basic sequence satisfying $\|z_i\|\ge 1$,
then $\sup_i\|z_i^*\|$ is finite, and the conclusion follows.

\end{enumerate}

Note that analysis of the proof of \cite{Pel62, Sin62} leads to
reasonably small absolute bounds for $\sup_i\|z_i^*\|$ and $C$
from above, but we do not need such bounds.
\medskip

Let us show that Lemma \ref{L:DerExt}, implies that to prove
Theorem \ref{T:Successor} it suffices to find a convex subset
$A\subset Z^*$ such that ${\overline{A}}^*=A^{(\alpha+1)}\neq
A^{(\alpha)}$. In fact, if we construct such $A$, we let
$D=(E^*)^{-1}(A)$. We have, by Lemma \ref{L:DerExt},
$D^{(\alpha)}=(E^*)^{-1}(A^{(\alpha)})$ and
$D^{(\alpha+1)}=(E^*)^{-1}(A^{(\alpha+1)})$. Since each functional
has a continuous extension, $A^{(\alpha+1)}\neq A^{(\alpha)}$
implies $D^{(\alpha+1)}\neq D^{(\alpha)}$.
\medskip

To show that ${\overline{A}}^*=A^{(\alpha+1)}$ implies
${\overline{D}}^*=D^{(\alpha+1)}$ we observe that
${\overline{A}}^*=A^{(\alpha+1)}$ implies that
$A^{(\alpha+1)}=A^{(\alpha+2)}$. By Lemma \ref{L:DerExt} the last
equality implies $D^{(\alpha+1)}=D^{(\alpha+2)}$. By the
Krein-Smulian theorem \cite[p.~429]{DS58}), the condition
$D^{(\alpha+1)}=D^{(\alpha+2)}$ implies
${\overline{D}}^*=D^{(\alpha+1)}$.

For this reason from now on our goal is to prove Theorem
\ref{T:Successor} for $X=Z$.
\medskip

\section{Construction of suitable
convex sets}\label{S:Sets}

We introduce an injective map of $F_\alpha$ into $\mathbb{N}$ and
identify each vertex of $F_\alpha$ with its image in $\mathbb{N}$.
We may and shall assume that if $x<y$ in $F_\alpha$, then the
images $\bar x$ and $\bar y$ of $x,y$ in $\mathbb{N}$ satisfy
$\bar x<\bar y$. In fact, to establish such identification we
reserve for each component of $F_\alpha$ an infinite subsequence
in $\mathbb{N}$ (we reserve the whole $\mathbb{N}$ if $F_\alpha$
is connected). Then we assign to the initial vertex of each of the
components of $F_\alpha$ the least number of the corresponding
subsequence and delete the initial vertices from $F_\alpha$.
Unless an initial vertex in a component $K$ of $F_\alpha$ was a
terminal vertex, deletion of it will split $K$ into infinitely
many (incomparable with respect to the partial order) components.
We reserve for each of these components a subsequence of the
sequence reserved for $K$, and continue in an obvious way.
\medskip

For each terminal vertex $v\equiv n_k\in \mathbb{N}$ (the symbol
$\equiv$ means that we identify the vertex $v$ and number $n_k$)
of $F_\alpha$ we consider the path joining $v$ with the initial
vertex of the component of $F_\alpha$ containing $v$. Let the path
be $n_k,\dots,n_1$. In the path we list vertices only and observe
that (by the result of the previous paragraph) $n_k,\dots, n_1$ is
a decreasing sequence in $\mathbb{N}$.

Introduce for a terminal vertex $v$ of $F_\alpha$, corresponding
to $\{n_k,\dots,n_1\}$, a vector in $Z^*$ given by
\begin{equation}\label{E:VericesA} z^*(v)=\sum_{i=1}^kn_{i-1}z^*_{n_i},
\end{equation} where we set $n_0=n_1$.

Let \begin{equation}\label{E:X} X=X_\alpha=\{z^*(v):~v\hbox{ is a
terminal vertex in }F_\alpha\},\end{equation} and let
$A=A_\alpha=\conv (X)$.

Our next goal is to analyze  the structure of the weak$^*$ derived
sets of $A$.

In this connection, for each countable ordinal $\beta$, we define
a \emph{shortening} $X^\beta$ of the set $X$  as
\[X^\beta=\left\{\sum_{n_i\in F_\alpha^\beta}n_{i-1}z^*_{n_i}:\quad
\sum_{i=1}^kn_{i-1}z^*_{n_i}\in X\right\},\] where $n_i\in
F_\alpha^\beta$ means that the vertex of $F_\alpha$ corresponding
to $n_i$ belongs to the (defined in Section \ref{S:Trees&Forests})
derived forest $(F_\alpha)^\beta$ of order $\beta$. We let
$X^0=X$.

\begin{remark}\label{R:v(y)} Observe that each vector $y$ in any of $X^\beta$ including
$X^0$ is supported on the vertex set of a finite path in
$F^\alpha$ whose vertex set is linearly ordered (recall that the
vertex set of $F^\alpha$ is partially ordered). We denote the
largest vertex in this path by $v(y)$.\end{remark}

\begin{corollary}[Of Lemma \ref{L:Der}]\label{C:NonTrDiff} The
difference
$X^{\beta+1}\backslash\left(\bigcup_{0\le\gamma\le\beta}
X^\gamma\right)$ is nonempty for every $\beta<\alpha$.
\end{corollary}

In fact, Lemma \ref{L:Der} implies that $X^{\beta+1}$ contains
some ``shortened'' vectors which are not in $X^\beta$.

\begin{corollary}[Of Lemma \ref{L:UpTerm}]\label{C:NoSubvec} It is
impossible for the support of a vector $y\in X^{\beta+1}$ to
contain the support of any of the vectors
$z\in\left(\bigcup_{0\le\gamma\le\beta} X^\gamma\right)$
\end{corollary}

\begin{proof}
In fact, if the support of $z$ is contained in the support of $y$,
then, on one hand, the vertex $v(z)$ has infinitely many terminal
up-neighbors in some $X^\gamma$, $\gamma<\beta$. On the other
hand, it means that not all of up-neighbors of $v(z)$ are in that
$X^\gamma$, because otherwise in could not happen that $y\in
X^{\beta+1}$. We get a contradiction with Lemma \ref{L:UpTerm}.
\end{proof}

\section{Proof of Theorem \ref{T:Successor}}\label{S:PrfThrm}

The main steps in our proof of Theorem \ref{T:Successor} are the
following:

\begin{enumerate}[{\bf (A)}]

\item \label{I:A} For every $\beta\le\alpha$

\begin{equation}\label{E:ContConv}\conv\left(\bigcup_{0\le \gamma\le \beta}
X^\gamma\right)\subset A^{(\beta)}.\end{equation}

\item \label{I:B} For every $\beta\le\alpha$

\begin{equation}\label{E:InclBeta}
A^{(\beta)}\subset\overline{\conv\left(\bigcup_{0\le \gamma\le
\beta} X^\gamma\right)}.\end{equation}

\item\label{I:C} If $\beta<\alpha$, then $X^{\beta+1}\backslash\,
\overline{\conv\left(\bigcup_{0\le \gamma\le \beta}
X^\gamma\right)}\ne\emptyset$.

\item\label{I:D} The weak$^*$ sequential closure of
$\overline{\conv\left(\bigcup_{0\le \gamma\le \alpha}
X^\gamma\right)}$ coincides with
$\overline{\conv\left(\bigcup_{0\le \gamma\le \alpha}
X^\gamma\right)}$. Therefore $\overline{\conv\left(\bigcup_{0\le
\gamma\le \alpha} X^\gamma\right)}$ is weak$^*$ closed.

\item \label{I:E} The inclusion in \eqref{E:InclBeta} becomes an
equality if $\alpha$ is a successor ordinal and $\beta=\alpha$.

\end{enumerate}

\subsection{Proof of item \eqref{I:A}}

Since convexity is preserved under weak$^*$ sequential closures,
to prove \eqref{I:A} by induction it suffices to show that
$X^\beta\subset A^{(\beta)}$ for every $\beta\le\alpha$.

The inclusion $X^1\subset A^{(1)}$ can be derived from the
definitions as follows. The definitions imply that $y\in
X^1\backslash X^0$ if and only if there is an infinite sequence of
terminal vertices $\{v_n\}\subset F_\alpha$ having the same
down-neighbor $u$, such that $y=z^*(u)$ (we use for non-terminal
vertices the same notation as in \eqref{E:VericesA}). Let
$x_n=z^*(v_n)$. Then, as it is easy to see, $x_n\in X$ and
$y=w^*-\lim_{n\to\infty} x_n$.

In a similar way, if we know that $X^\beta\subset A^{(\beta)}$ we
derive $X^{\beta+1}\subset A^{(\beta+1)}$ from the fact that each
element of $X^{\beta+1}\backslash X^\beta$ is a limit of a
weak$^*$ convergent sequence of elements of $X^\beta$.

On limit ordinals. The definition
$(F_\alpha)^\beta=\bigcap_{\gamma<\beta}(F_\alpha)^\gamma$ for the
derived forest of order $\beta$ with a limit ordinal $\beta$
implies that
\begin{equation}\label{E:XbetaLimit}
X^\beta\subset \bigcup_{\gamma<\beta}X^\gamma
\end{equation}
for a limit ordinal $\beta$. Combining inclusion
\eqref{E:XbetaLimit} with the definition of the weak$^*$ derived
set $A^{(\beta)}$ for a limit ordinal $\beta$ (see
\eqref{E:DefLimDer}), we get that the validity of inclusion
\eqref{E:ContConv} for all ordinals $\tau<\beta$ implies its
validity for a limit ordinal $\beta$.

\subsection{Proof of item \eqref{I:B}}\label{S:Beta=1}

We prove \eqref{E:InclBeta} by induction. Of course, we have the
inclusion for $\beta=0$.

The next step is to suppose that we have
\[A^{(\beta)}\subset\overline{\conv\left(\bigcup_{0\le \gamma\le
\beta} X^\gamma\right)},\] and to use this inclusion to derive
\[A^{(\beta+1)}\subset\overline{\conv\left(\bigcup_{0\le \gamma\le
\beta+1} X^\gamma\right)}.\]

To achieve this it is clearly enough to show that
\[\left(\conv\left(\bigcup_{0\le \gamma\le
\beta}
X^\gamma\right)\right)^{(1)}\subset\overline{\conv\left(\bigcup_{0\le
\gamma\le \beta+1} X^\gamma\right)}.
\]

Proof of the step $\beta\to \beta+1$ will complete the proof of
\eqref{E:InclBeta}, because for a limit ordinal $\beta$ the
inclusion \eqref{E:InclBeta} follows immediately from the
definition of $A^{(\beta)}$ for a limit ordinal $\beta$, provided
\eqref{E:InclBeta} has been already proved for all $\tau<\beta$.
\medskip

So we prove the step $\beta\to\beta+1$.

Since the set $\conv\left(\bigcup_{0\le \gamma\le \beta}
X^\gamma\right)$ is a subset of the dual of a separable Banach
space, any element of its weak$^*$ derived set is a weak$^*$ limit
of a bounded sequence of the form
\begin{equation}\label{E:SeqConvC}
\left\{\sum_{x\in W}a_{x,i}x\right\}_{i=1}^\infty,\hbox{ where }
a_{x,i}\ge 0,\quad \sum_{x\in W}a_{x,i}=1,
\end{equation}
where $W=\bigcup_{0\le \gamma\le \beta} X^\gamma$ and the set
$\{a_{x,i}\}_{x\in W}$ is finitely nonzero for any
$i\in\mathbb{N}$.

For each $x\in W$ we consider the vertex $v(x)$ in $F_\alpha$ (see
the definition in Remark \ref{R:v(y)}). It can happen that for
some $x\in W$ the vertex $v(x)$ is an initial vertex. We denote
the set of all such $x\in W$ by $I$.

For $x\in(W\backslash I)$ denote by $v(y)$ the down-neighbor of
$v(x)$ in $F_\alpha$ and denote by $y=y(x)$ the vector in $Z^*$
obtained if we replace by $0$ the component of $x$ corresponding
to $v(x)$, so that $v(y)$ agrees with the definition in Remark
\ref{R:v(y)}.

We group the summands of $\sum_x a_{x,i}x$ for $x\in(W\backslash
I)$ according to vectors $y=y(x)$ defined in the previous
paragraph. We denote the set of all such vectors $y$ obtained for
different $x\in (W\backslash I)$  by $D$. We can write
\[\sum_{x\in W} a_{x,i}x=\sum_{x\in I} a_{x,i}x+\sum_{y\in D}~~\sum_{\{x:\, v(x)\succ v(y)\}}
a_{x,i}x,\] where $v(x)\succ v(y)$ means that $v(y)$ is a
down-neighbor of $v(x)$.

We may assume without loss of generality that $\lim_{i\to\infty}
\sum_{\{x:\, v(x)\succ v(y)\}} a_{x,i}$ exists for every $y\in D$
and denote this limit by $s_y$. We may also assume that
$\lim_{i\to\infty} a_{x,i}$ exists for every $x\in W$ and denote
this limit by $p_x$.

\begin{lemma}\label{L:Sum1} If the sequence $\left\{\sum_{x\in W} a_{x,i}x\right\}_{i=1}^\infty$ is bounded, then $\sum_{x\in I}p_x+\sum_{y\in D}
s_y=1$.
\end{lemma}

\begin{proof} In fact, suppose $\sum_{x\in I}p_x+\sum_{y\in D}
s_y=\omega<1$. For every pair of finite subsets $G\subset I$ and
$F\subset D$, and every $\ep>0$, there is $j\in\mathbb{N}$ such
that
\[\sum_{x\in G}(a_{x,i}-p_x)+\sum_{y\in F}\left(\left(\sum_{\{x:\, v(x)\succ v(y)\}}
a_{x,i}\right)-s_y\right)<\ep \hbox{ for }i\ge j.\]

Therefore

\[\sum_{x\in (I\backslash G)}a_{x,i}+\sum_{y\in (D\backslash F)}\left(\sum_{\{x:\, v(x)\succ v(y)\}}
a_{x,i}\right)>1-(\omega+\ep) \hbox{ for }i\ge j.\]

For any $M\in \mathbb{N}$, we can pick $F$ in such a way that for
all $y\in (D\backslash F)$ the natural number corresponding to
$v(y)$ in the identification described in Section \ref{S:Sets} is
at least $M$. (For this and the next statement we need to recall
that $n_0=n_1$ in \eqref{E:VericesA}.)

Similarly, we can pick $G$ in such a way that for all $x\in
(I\backslash G)$, the natural number corresponding to $x$ (recall
that $n_0=n_1$, see the line after \eqref{E:VericesA}) is at least
$M$. Then, by {\bf \eqref{I:Le}} in Section \ref{S:RedSeparable},
\[\begin{split}\left\|\sum_{x\in W} a_{x,i}x\right\|&\ge
\frac1C\,z^{**}\left(\sum_{x\in W} a_{x,i}x\right)\\&\ge
\frac1C\,z^{**} \left(\sum_{x\in (I\backslash
G)}a_{x,i}x+\sum_{y\in (D\backslash F)}\left(\sum_{\{x:\,
v(x)\succ v(y)\}} a_{x,i}x\right)\right)\\&\ge
\frac{M(1-(\omega+\ep))}{C}.\end{split}\] Since this can be done
for every $M\in\mathbb{N}$ and every $\ep>0$, we conclude that
$\left\{\sum_{x\in W} a_{x,i}x\right\}_{i=1}^\infty$ is unbounded.
This contradiction proves the lemma.\end{proof}

\begin{lemma}\label{L:StrInL1} If $\sum_{x\in I}p_x+\sum_{y\in D} s_y=1$, then the vectors $\left\{\sum_{\{x:\, v(x)\succ v(y)\}}
a_{x,i}\right\}_{y\in D}$ converge to the vector $\{s_y\}_{y\in
D}$ strongly in $\ell_1(D)$ and the vectors $\{a_{x,i}\}_{x\in I}$
converge to $\{p_x\}_{x\in I}$ in $\ell_1(I)$.
\end{lemma}

Lemma \ref{L:StrInL1} is an immediate consequence of the fact that
a sequence of normalized vectors $\{v_i\}$ in $\ell_1$ which
converges pointwise to a normalized vector $v$, converges to $v$
strongly.

\begin{lemma}\label{L:ConvWp_x} The series $\sum_{x\in W} p_xx$ and $\sum_{x\in I}
p_xx$ are strongly convergent.
\end{lemma}

\begin{proof} In fact, otherwise by items \eqref{I:Le} and \eqref{I:Ge} in Section \ref{S:RedSeparable}, the series
 $\sum_{x\in W} p_xz^{**}(x)$ diverges to infinity. This divergence
 implies that $\sum_{x\in W} a_{x,i} z^{**}(x)\le C\left\|\sum_{x\in W}
 a_{x,i} x\right\|$ tend to infinity as $i\to\infty$ contradicting
 the boundedness of the sequence $\left\{\sum_{x\in W} a_{x,i}x\right\}_{i=1}^\infty$
\end{proof}

\begin{corollary} If the sequence $\left\{\sum_{x\in W} a_{x,i}x\right\}_{i=1}^\infty$ is bounded in $Z^*$, then the sequence $\{\sum_{x\in I}
a_{x,i}x\}_{i=1}^\infty$ weak$^*$ converges to the vector
$\sum_{x\in I}p_x x$.
\end{corollary}

\begin{proof} In fact, the mentioned sequence is uniformly bounded
and convergent coordinate-wise (this is very easy to see because
$x\in I$ are vectors with one nonzero coordinate each). By the
well-known simple fact, they converge in the weak$^*$ topology.
\end{proof}

\begin{lemma}\label{L:SeriesD} The series $\sum_{y\in D} s_y y$, and thus $\sum_{y\in
D}\left(s_y-\sum _{v(x)\succ v(y)}p_x\right)y$, is strongly
convergent.\end{lemma}

\noindent The first statement implies the second statement by
virtue of an easy inequality $s_y\ge\sum _{x:~v(x)\succ v(y)}p_x$.

\begin{proof} Assume the contrary. Since the vectors $y$ are nonnegative and $s_y\ge 0$,
the contrary, by \eqref{I:Le} and \eqref{I:Ge} on page
\pageref{I:Ge}, implies that $\sum_y s_yz^{**}(y)$ diverges to
$\infty$. On the other hand, for each $y\in D$ and sufficiently
large $i=i(y)$ we have
\[z^{**}\left(\sum_{x:~v(x)\succ v(y)} a_{x,i}x\right)\ge \frac12\,s_yz^{**}(y).\]
Since the sets $\{x:~v(x)\succ v(y)\}$ with different $y$ are
disjoint, we conclude that sums $\sum_{x\in W} a_{x,i}x$ cannot be
uniformly bounded.
\end{proof}

\begin{lemma} The sequence of vectors
\begin{equation}\label{E:A_iy}\sum_{x\in W} a_{x,i}x=\sum_{x\in I} a_{x,i}x+\sum_{y\in
D}~~\sum_{\{x:\, v(x)\succ v(y)\}} a_{x,i}x,~ i\in
\mathbb{N},\end{equation} converges to the vector
\begin{equation}\label{E:B_y} \sum_{x\in I} p_xx +\sum_{y\in D}\left(\sum_{\{x:\, v(x)\succ v(y)\}} p_x x
+\left(s_y-\sum_{\{x:\, v(x)\succ v(y)\}} p_x\right)y\right)
\end{equation}
in the weak$^*$ topology.
\end{lemma}

\begin{proof} We know that vectors \eqref{E:A_iy} are uniformly
bounded. Because of this it is enough to prove that the vectors
\eqref{E:A_iy} converge to the vector \eqref{E:B_y} componentwise.

Let us consider their $m^{\rm th}$ components. Assume that $m$ is
in the image of $F_\alpha$ and that the path to $m$ from the
initial vertex of the component containing $m$ is
$n_1,\dots,n_k=m$. There are two slightly different cases: $k=1$
and $k\ge 2$. We consider the case $k\ge 2$. The change which
should be made if $k=1$ is: replace $n_{k-1}$ by $n_1$ in the
formulas below and add the corresponding term for $x\in I$.

In the case $k\ge 2$ the $m^{\rm th}$ component of the vector
\eqref{E:B_y} is

\[p_zn_{k-1}+\sum_{\stackrel{y\in D}{m\in\supp y}}s_yn_{k-1},
\]
where in the first term, $z$ is such that $m=v(z)$. On the other
hand, the $m^{\rm th}$ components of the vectors \eqref{E:A_iy}
are
\[a_{z,i}n_{k-1}+\sum_{\stackrel{y\in D}{m\in\supp y}}\sum_{x:~v(x)\succ v(y)}a_{x,i}n_{k-1},
\]
where in the first term $z$ is such that $m=v(z)$.

We have the convergence of $\sum_{\stackrel{y\in D}{m\in\supp
y}}\sum_{x:~v(x)\succ v(y)}a_{x,i}n_{k-1}$ to
$\sum_{\stackrel{y\in D}{m\in\supp y}}s_yn_{k-1}$ by Lemma
\ref{L:StrInL1}.
\end{proof}

To complete the proof of item \eqref{I:B}, it suffices:

\begin{enumerate}[{\bf (1)}]

\item To recall that (see Lemma \ref{L:Sum1})

\[\sum_{x\in I}p_x+\sum_{y\in D}\left(\sum_{\{x:\, v(x)\succ v(y)\}}
p_x\right)+\sum_{y\in D}\left(s_y-\sum_{\{x:\, v(x)\succ v(y)\}}
p_x\right)=1.\]

\item To derive from the previous item and Lemmas \ref{L:ConvWp_x}
and \ref{L:SeriesD} that the vector in \eqref{E:B_y} is an
infinite convergent convex combination of $x\in W$ and $y\in D$.

\item To observe that $y\in D$ can be involved in this combination
with nonzero coefficient only if $\left(s_y-\sum_{\{x:\, v(x)\succ
v(y)\}} p_x\right)>0$, and this can happen only if
$y\in\bigcup_{0\le \gamma\le \beta+1} X^\gamma$.

\end{enumerate}

\subsection{Proof of item \eqref{I:C}}

If $\beta<\alpha$, some vertices of $(F_\alpha)^\beta$ are not in
$(F_\alpha)^{\beta+1}$ (see Lemma \ref{L:Der}) and therefore there
is an infinite family of terminal vertices of $(F_\alpha)^\beta$
with the common down-neighbor $u$. Then there is $y\in
X^{\beta+1}\backslash\left(\bigcup_{0\le \gamma\le \beta}
X^\gamma\right)$ satisfying $v(y)=u$ (see Section \ref{S:Sets}).

To complete the proof of \eqref{I:C} it suffices to show that
$y\notin\overline{\conv\left(\bigcup_{0\le \gamma\le \beta}
X^\gamma\right)}$.

Note that $\bigcup_{0\le \gamma\le \beta} X^\gamma$ contains
vectors of two types:

\begin{enumerate}[{\bf (1)}]

\item Extensions of $y$, that is, vectors coinciding with $y$ on
its support, but also having at least one more positive
coordinate. According to the definitions in Section \ref{S:Sets},
the coordinate has to be $\ge m$, where $m$ is the image of $u$ in
$\mathbb{N}$, see the beginning of Section \ref{S:Sets}. We denote
this set of all such vectors in $\bigcup_{0\le \gamma\le \beta}
X^\gamma$ by $E$.

\item Vectors whose $m^{\rm th}$ coordinate is equal to $0$ and
some coordinates which are not in the support of $y$ are positive.
We denote the set of all such vectors in $\bigcup_{0\le \gamma\le
\beta} X^\gamma$ by $R$. By Corollary \ref{C:NoSubvec}, all
vectors in $R$ have nonzero coordinates which are not in the
support of $y$.

\end{enumerate}

Clearly,
\[\overline{\conv\left(\bigcup_{0\le \gamma\le \beta}
X^\gamma\right)}=\overline{\conv (E\cup R)}.\]

So we need to find a continuous linear functional on $Z^*$ which
separates $y$ from $\conv (E\cup R)$. In this connection, we
consider the following two continuous linear functionals on $Z^*$.

The first is the sum of all coordinates of a vector with respects
to the basis $\{z_i^*\}_{i=1}^\infty$, except the coordinates
which are nonzero for $y$. This functional is continuous because
it is a linear combination of $z^{**}$ (introduced in Section
\ref{S:RedSeparable}) and finitely many functionals of the
sequence $\{z_i\}_{i=1}^\infty$ considered as elements of
$Z^{**}$. We denote this functional by $\tilde z$.

The second functional is $z_m$ (that is, $m^{\rm th}$ coordinate
functional).

We claim that $z_m-\tilde z$ separates $y$ from $\conv (E\cup R)$.
To see this observe that $(z_m-\tilde z)(y)=a>0$, where $a$ is the
value of the largest coordinate of $y$.

On the other hand, $(z_m-\tilde z)|_R\le 0$ because $z_m$ is zero
on $R$ and $\tilde z$ is nonnegative for all vectors in $R$.

Also $(z_m-\tilde z)|_E\le 0$ because for each vector in the
extension further coordinates cannot be smaller than the previous
ones.

\subsection{Proof of item \eqref{I:D}}

To proof this statement, we repeat the argument used to prove
\eqref{I:B} and observe that we get into the closure of the same
set because $D$, in this case, is a subset of $W$. By the
Krein-Smulian theorem \cite[p.~429]{DS58}), this completes the
proof.

\subsection{Proof of item \eqref{I:E}}

Let $\alpha=\tau+1$. By item {\bf (A)}, we have
\[A^{(\tau)}\supset\conv\left(\bigcup_{1\le\gamma\le\tau}
X^\gamma\right).\]

Since the weak$^*$ derived set (for any set $A$) contains the
strong closure of the set, we get
\[A^{(\tau+1)}\supset\overline{\conv\left(\bigcup_{1\le\gamma\le\tau}
X^\gamma\right)}.\]

Denote by $j$ the initial vertex of the tree $F_{\tau+1}$, and let
$r\in Z^*$ be the vector whose only nonzero coordinate in
$\{z_i^*\}_{i=1}^\infty$ is for $k\in\mathbb{N}$ which corresponds
to $j$ according to the injective map constructed at the beginning
of Section \ref{S:Sets}. So $r=kz^*_k$, the definition
\eqref{E:VericesA} takes this form in the case where $j$ is the
initial vertex of $F_{\tau+1}$. By Lemma \ref{L:Der} and item
\eqref{I:A} we have $r\in A^{(\tau+1)}$. Our goal is to prove

\[A^{(\tau+1)}\supset\overline{\conv\left(\left(\bigcup_{1\le\gamma\le\tau}
X^\gamma\right)\bigcup\{r\}\right)}.\]

So we consider a strongly convergent sequence
\[\{\sum_{x\in W}a_{x,i}x+a_{r,i}r\}_{i=1}^\infty,\]
where $W=\bigcup_{1\le\gamma\le\tau} X^\gamma$, $a_{x,i}\ge 0$,
$a_{r,i}\ge 0$, and $\sum_{x\in W}a_{x,i}+a_{r,i}=1$. Since $0\le
a_{r,i}\le 1$, we may assume that the sequence
$\{a_{r,i}r\}_{i=1}^\infty$ is convergent. Since $A^{(\tau+1)}$ is
convex, the conclusion follows.

\subsection{End of the proof of Theorem
\ref{T:Successor}}\label{S:End}

It is clear that combination of items \eqref{I:A}-\eqref{I:E}
proves Theorem \ref{T:Successor} in all cases except case
$\alpha=1$. Our proof can be adjusted to cover this case - just
consider a forest consisting of infinitely many disjoint copies of
$F_1$. We do not provide the details because the case $\alpha=1$
is covered by Silber \cite{Sil21}.
\end{proof}

\subsection{Comment}

The following question asked in \cite{Sil21} remains unanswered.

\begin{question}[{\cite[Section 3, Question 1]{Sil21}}]\label{Q:Limit} Does there exist a convex subset $A$ in the
dual to a separable Banach space for which the least ordinal
$\alpha$ satisfying $A^{(\alpha)}={\overline{A}\,}^*$ is a limit
ordinal?
\end{question}

It should be mentioned that it is an easy consequence of the Baire
theorem that this cannot happen if we additionally require that
$A$ is subspace. To the best of our knowledge, Godun \cite{God77}
was the first to make this observation (later it was repeatedly
rediscovered, see \cite{KL87}, \cite{HS96}).

\section*{Acknowledgement}

The author gratefully acknowledges the support by the National
Science Foundation grant NSF DMS-1953773.

\textsc{Department of Mathematics and Computer Science, St. John's
University, 8000 Utopia Parkway, Queens, NY 11439, USA} \par
  \textit{E-mail address}: \texttt{ostrovsm@stjohns.edu} \par

\end{large}
\end{document}